\documentclass{amsart} 
\usepackage{amsmath,amssymb,amsthm,amsfonts,amsxtra,mathtools}
\usepackage{enumerate,verbatim,mathrsfs,comment,color,url}
\usepackage[all,2cell,ps]{xy}
\usepackage[pagebackref]{hyperref}
\usepackage{graphicx}
\usepackage{verbatim}

\DeclareMathOperator{\ann}{ann}
\DeclareMathOperator{\Ass}{Ass}

\DeclareMathOperator{\codim}{codim}

\DeclareMathOperator{\coker}{Coker}

\DeclareMathOperator{\depth}{depth}

\DeclareMathOperator{\Ext}{Ext}

\DeclareMathOperator{\gdim}{G-dim}

\DeclareMathOperator{\Hom}{Hom}

\DeclareMathOperator{\id}{id}
\DeclareMathOperator{\Image}{Image}

\DeclareMathOperator{\injdim}{injdim}

\DeclareMathOperator{\pd}{pd}

\DeclareMathOperator{\projdim}{projdim}

\DeclareMathOperator{\rank}{rank}

\DeclareMathOperator{\Spec}{Spec}

\DeclareMathOperator{\Supp}{Supp}

\DeclareMathOperator{\type}{type}

\renewcommand{\ge}{\geqslant}
\renewcommand{\le}{\leqslant}

\newcommand{\fm}{\mathfrak{m}}
\newcommand{\fp}{\mathfrak{p}}
\renewcommand{\iff}{if and only if }

\theoremstyle{plain}
\newtheorem{theorem}{Theorem}[section]

\newtheorem{proposition}[theorem]{Proposition}
\newtheorem{corollary}[theorem]{Corollary}

\newtheorem{innercustomtheorem}{Theorem}
\newenvironment{customtheorem}[1]
  {\renewcommand\theinnercustomtheorem{#1}\innercustomtheorem}
  {\endinnercustomtheorem}

\theoremstyle{definition}

\newtheorem{conjecture}[theorem]{Conjecture}

\newtheorem{example}[theorem]{Example}

\newtheorem{para}[theorem]{}
\newtheorem{question}[theorem]{Question}
\newtheorem{setup}[theorem]{Setup}

\theoremstyle{remark}
\newtheorem{remark}[theorem]{Remark}

\numberwithin{equation}{section}

\title[Auslander-Reiten conjecture and injective dimension of Hom]{Auslander-Reiten conjecture and finite injective dimension of Hom}

\author[Dipankar Ghosh]{Dipankar Ghosh}
\address{Department of Mathematics, Indian Institute of Technology Kharagpur, West Bengal - 721302, India}
\email{dipankar@maths.iitkgp.ac.in, dipug23@gmail.com}
\author[Ryo Takahashi]{Ryo Takahashi}
\address{Graduate School of Mathematics, Nagoya University, Furocho, Chikusaku, Nagoya 464-8602, Japan}
\email{takahashi@math.nagoya-u.ac.jp}

\date{September 2, 2021}

\begin{document}

\pagenumbering{arabic}
\thispagestyle{empty}
  
 \begin{abstract}
 	For a finitely generated module $ M $ over a commutative Noetherian ring $R$, we settle the Auslander-Reiten conjecture when at least one of $\Hom_R(M,R)$ and $\Hom_R(M,M)$ has finite injective dimension. A number of new characterizations of Gorenstein local rings are also obtained in terms of vanishing of certain Ext and finite injective dimension of Hom.
 \end{abstract}
\maketitle
\section{Introduction}
	
	\begin{setup}\label{setup}
		Unless otherwise specified, $R$ is a commutative Noetherian local ring of dimension $ d $. All $R$-modules are assumed to be finitely generated.
	\end{setup}
	
	The vanishing of Ext modules, and their consequences are actively studied subjects in commutative algebra. A bunch of criteria for a given module to be projective, and criteria for a local ring to be Gorenstein have been described in terms of the vanishing of Ext. The purpose of this article is to provide such criteria in terms of properties of $\Hom_R(M,N)$ for $R$-modules $M$ and $N$ when the vanishing of $\Ext_R^{1 \le i \le n}(M,N)$ is given for some $n\ge 1$. We prove the following results.
	
	\begin{customtheorem}{\ref{thm:Hom-injdim-finite-consequence}}\label{thm:main-1}
		Let $M$ and $N$ be nonzero $R$-modules such that $\depth(N)=d$ and $\Ext_R^i(M,N)=0$ for all $1 \le i \le d$. Then $\Hom_R(M,N)$ has finite injective dimension \iff $M$ is free and $N$ has finite injective dimension.
	\end{customtheorem}
	
	\begin{customtheorem}{\ref{thm:main-2}}
		Let $M$ be a nonzero $R$-module such that $ \Ext_R^i(M,R) = 0$ and $ \Ext_R^j(M,M) = 0 $		for all $ 1 \le i \le 2d+1 $ and $ 1 \le j \le \max\{1, d-1\} $. If $\Hom_R(M,M)$ has finite injective dimension, then $M$ is free, and $R$ is Gorenstein.
	\end{customtheorem}
	
	One of the most celebrated long-standing conjectures in commutative algebra is the Auslander-Reiten conjecture:
	
	\begin{conjecture}\cite{AR75}
		For an $R$-module $M$, if $\Ext_R^i(M, M\oplus R) = 0$ for all $i\ge 1$, then $M$ is projective.
	\end{conjecture}

	The conjecture is known to hold true in the following cases:
	(1) $R$ is complete intersection \cite[1.9]{ADS93}. (2) $M$ has finite complete intersection dimension \cite[Thm.~4.3]{AY98}.
	(3) $R$ is a deformation of a CM local ring of minimal multiplicity \cite{DG21}.
	(4) $ R $ is a locally excellent Cohen-Macaulay (in short, CM) normal ring containing $\mathbb{Q}$ \cite[Thm.~0.1]{HL04}. (5) $R$ is a Gorenstein normal ring \cite[Cor.~4]{Ar09}. (6) $R$ is a fiber product of two local rings of the same residue field \cite[1.2]{NS17}.	(7) $R$ is CM of dimension $d \ge 1$, $M$ is maximal Cohen-Macaulay (in short, MCM) such that $\Ext_R^{1 \le i \le d}(M, \Hom_R(M,M)) = 0$ and $M$ has rank $1$ \cite[Thm.~1.5]{GT17}. (8) $R$ is a CM normal domain, $M$ is MCM and $\Hom_R(M,M)$ is free \cite[Thm.~3.16]{DEL21}. (9) $R$ is normal and $\Hom_R(M,M)$ has finite Gorenstein dimension \cite[Cor.~1.6]{ST19}. (10) $R$ is CM normal \cite[Cor.~1.3]{KOT21}. (11) $R$ is CM such that $e(R)\le (7/4) \codim(R) + 1$, or $R$ is Gorenstein such that $ e(R) \le \codim(R) + 6$ \cite[Thm.~C]{LM20}. (There is some overlapping among these eleven conditions, e.g., (1) is included in (2) as well as in (3); while (4), (5) and (8) are included in (10).) However the conjecture is widely open even for Gorenstein local rings. In the present study, as applications of Theorems~\ref{thm:main-1} and \ref{thm:main-2}, we obtain the following.
	\begin{corollary}[=\ref{cor:M-star-ARC} and \ref{cor:ARC}]
		The Auslander-Reiten conjecture holds true for a $($finitely generated$)$ module $M$ over a commutative Noetherian ring $R$ when at least one of $\Hom_R(M,R)$ and $\Hom_R(M,M)$ has finite injective dimension.
	\end{corollary}
	More precise statements (about criteria for a module to be free over a local ring) can be seen in Corollaries~\ref{cor:injdim-M*-Gor-criteria}.(2), \ref{cor:freeness-self-dual-injdim} and \ref{cor:Gor-ARC}, and Theorem~\ref{thm:main-2}.
	
	We also provide some new characterizations of Gorenstein local rings in terms of vanishing of certain Ext and finite injective dimension of Hom.
	
	\begin{customtheorem}{\ref{thm:characterizations-Gor}}\label{cor:char-Gor-rings}
		The following statements are equivalent:
		\begin{enumerate}[\rm (1)]
			\item $R$ is Gorenstein.
			\item $R$ admits a module $M$ such that $\Ext_R^i(M,R)=0$ for all $ 1 \le i \le d-1 $ and $ M^* $ is nonzero and of finite injective dimension.
			\item $R$ admits a nonzero module $M$ such that $\Ext_R^i(M,R)=\Ext_R^j(M,M)=0$ for all $1\le i\le 2d+1$ and $1\le j\le\max\{1,d-1\}$, and $\injdim_R(\Hom_R(M,M)) < \infty$.
			\item $R$ admits a module $M$ such that $\depth(M)=d$, $\Ext_R^j(M,M)=0$ for all $1 \le j \le d$ and $\injdim_R(\Hom_R(M,M)) < \infty$.
			\item $R$ admits a module $M$ such that
			\begin{center}
				$ \depth(\Hom_R(M,M)) = \depth(R) $ and $ \injdim_R(\Hom_R(M,M)) < \infty $.
			\end{center}
		\end{enumerate}
	\end{customtheorem}

	The characterizations given in Theorem~\ref{cor:char-Gor-rings} are motivated by the criteria for Gorenstein local rings due to Ulrich, Hanes-Huneke, Jorgensen-Leuschke (see \ref{para:char-Gor} for the details), and the following classical results.

	\begin{theorem}[Peskine-Szpiro]\cite[Chapitre II, Th\'{e}or\`{e}me (5.5)]{PS73}\label{PS}
		The ring $R$ is Gorenstein \iff it has a nonzero cyclic module of finite injective dimension.
	\end{theorem}

	\begin{theorem}[Foxby]\cite{Fo77}\label{Foxby}
		The ring $R$ is Gorenstein \iff it has a nonzero $($finitely generated$)$ module $M$ for which $\injdim_R(M) < \infty$ and $\projdim_R(M) < \infty$.
	\end{theorem}
	
	\begin{theorem}[Bass' Conjecture]\cite{PS73, Rob87}\label{Bass}
		If $R$ admits a nonzero $($finitely generated$)$ module of finite injective dimension, then $R$ is CM.
	\end{theorem}

\section{Criteria for a module to be free}

\begin{para}
	Let $M$ be an $R$-module. Set $M^* := \Hom_R(M,R)$. The minimal number of generators of $M$ is denoted by $\nu(M)$, i.e., $\nu(M) = \dim_k(M \otimes_R k)$.
	The type of $M$ is defined to be $\type(M)=\dim_k(\Ext_R^t(k,M))$, where $t=\depth(M)$.
\end{para}


\begin{remark}\label{rmk:Hom-nonzero}
	If $M$ and $N$ are nonzero $R$-modules, then
	\begin{equation}
		0 \le \depth(\ann_R(M),N) = \min\{ i : \Ext_R^i(M,N) \neq 0 \} \le d.
	\end{equation}
	Hence, if $\Ext_R^{1 \le i \le d}(M,N)=0$, then $\Hom_R(M,N)$ is a nonzero module.
\end{remark}

%
%
%
%
%
 
 We deduce our first main result from the following theorem, where we study the consequences of $\Hom_R(M,N)$ having finite injective dimension under the condition that $\Ext^{1 \le i \le d-1}_R(M,N) = 0$.
 
\begin{theorem}\label{thm:Hom-injdim-finite}
	Let $M$ and $N$ be $R$-modules such that $\depth(N)=d$,
	\begin{center}
		$\Hom_R(M,N) \ne 0$	and $\Ext^i_R(M,N)=0$ for all $1 \le i \le d-1$.
	\end{center}
	Suppose that $\Hom_R(M,N)$ has finite injective dimension. Then, $R$ is CM, $N$ is MCM and of finite injective dimension, and $M\cong\Gamma_{\fm}(M) \oplus R^r$ for some $r\ge0$.
\end{theorem}

\begin{proof}
	By Theorem~\ref{Bass}, $R$ is CM, and hence $N$ is MCM. We need to show that
	\begin{equation}\label{claim}
	\mbox{$\injdim_R(N) < \infty$\, and $M\cong\Gamma_{\fm}(M)\oplus R^r$ for some $r\ge0$.}
	\end{equation}
	We may assume that $R$ is complete. Then $R$ admits a canonical module $\omega$.
	
	(1) We first prove the assertion \eqref{claim} when $N$ is indecomposable. Let $\mathbb{F} : \cdots \to F_2 \to F_1 \to F_0 \to 0 $ be a free resolution of $M$. Since $\Ext^i_R(M,N)=0$ for all $1 \le i \le d-1$, the resolution $\mathbb{F}$ induces an exact sequence
	\begin{equation*}
		0 \to \Hom_R(M,N) \to \Hom_R(F_0, N) \xrightarrow{f} \Hom_R(F_1, N) \to \cdots \xrightarrow{g} \Hom_R(F_e, N),
	\end{equation*}
	where $ e = \max\{1, d\} $. Set $ C := \Image(f) $ and $D:=\coker(g)$. Hence there are two exact sequences:
	\begin{align}
	&0 \to \Hom_R(M,N) \to \Hom_R(F_0, N) \to C \to 0 \quad \mbox{ and}\label{s.e.s-1-Hom}\\
	&0 \to C \to \Hom_R(F_1, N) \to \Hom_R(F_2, N) \to \cdots \xrightarrow{g} \Hom_R(F_e, N) \to D \to 0.\label{s.e.s-2-Hom}
	\end{align}
	Note that each $\Hom_R(F_i, N)$ is MCM. In the first case, assume that $C=0$. Then $\Hom_R(M,N) \cong \Hom_R(F_0, N)$ is MCM. In the second case, we have $C\neq 0$. Applying the depth lemma to the exact sequences \eqref{s.e.s-2-Hom} and \eqref{s.e.s-1-Hom} respectively, we obtain that both $C$ and $\Hom_R(M,N)$ are MCM. Thus, in any case, $\Hom_R(M,N) \cong \omega^n$ for some $n\ge 1$, cf.~\cite[3.3.28]{BH93}. Moreover, in both the cases,
	\begin{center}
		$\Ext^1_R(C,\Hom_R(M,N)) \cong \Ext^1_R(C,\omega^n) = 0$,
	\end{center}
	see, e.g., \cite[3.3.10]{BH93}. Therefore \eqref{s.e.s-1-Hom} splits. Hence, setting $m := \rank(F_0)$,
	\begin{center}
		$N^m \cong \Hom_R(F_0, N) \cong \Hom_R(M,N) \oplus C \cong \omega^n \oplus C$.
	\end{center}
	Since $N$ is indecomposable, by the Krull--Schmidt theorem, $N\cong \omega$. In particular, $\injdim_R(N)$ is finite.
	
	Next we show that $M':=M/\Gamma_{\fm}(M)$ is free, i.e., $M'\cong R^r$ for some $r\ge0$. It further implies that the exact sequence $0\to\Gamma_{\fm}(M)\to M\to M'\to0$ splits, and hence $M\cong\Gamma_{\fm}(M)\oplus R^r$. In order to prove that $M'$ is free, we may assume that $M'\neq 0$, i.e., $\Gamma_{\fm}(M) \neq M$. Particularly, it ensures that $d\ge 1$. We take two steps.
	
	(1a) Consider the case $\Gamma_{\fm}(M)=0$. Then the local duality theorem (cf.~\cite[3.5.8]{BH93}) implies that $\Ext^d_R(M,\omega)=0$. Hence, by the given assumption, $\Ext_R^i(M,\omega) \cong \Ext_R^i(M,N)=0$ for all $1\le i\le d$. Therefore $M$ is MCM (cf.~\cite[3.1.24]{BH93}) and
	\begin{center}
		$ M \cong M^{\dag\dag} \cong \Hom_R(M,N)^\dag \cong (\omega^n)^\dag \cong R^n $,
	\end{center}
	where $ (-)^\dag = \Hom_R(-,\omega) $; see, e.g., \cite[3.3.10]{BH93}. Thus $M' = M$ is free.
	
	(1b) Consider the general case. Since $\Gamma_{\fm}(M)$ has finite length,
	\begin{center}
		$\Ext^i_R(\Gamma_{\fm}(M),N) \cong \Ext^i_R(\Gamma_{\fm}(M),\omega) = 0 $ for all $i < d$.
	\end{center}
	Hence, from the long exact sequence of Ext modules induced by $0\to\Gamma_{\fm}(M)\to M\to M'\to0$, it follows that
	\begin{center}
		$\Hom_R(M',N)\cong \Hom_R(M,N)\ne 0$ and $\Ext^i_R(M',N) \cong \Ext^i_R(M,N) = 0$
	\end{center}
	for all $1\le i\le d-1$.  Since $\Gamma_{\fm}(M')=0$, we can apply (1a) to see that $M'$ is free. This completes the proof of \eqref{claim} when $N$ is indecomposable.
	
	(2) Next we consider the general case. Since $\Hom_R(M,N) \ne 0$, there exists an indecomposable direct summand $N'$ of $N$ such that $\Hom_R(M,N')\neq 0$. Applying (1) to $M$ and $N'$, we have that $N' \cong \omega$ and $M \cong \Gamma_{\fm}(M)\oplus R^r$ for some $r\ge 0$. It remains to show that $\injdim_R(N) < \infty$. We have two possible cases.
	
	(2a) Suppose that $r \ge 1$. Note that $\Hom_R(M,N) \cong \Hom_R(\Gamma_{\fm}(M),N) \oplus N^r$. Thus $N$ is a direct summand of $\Hom_R(M,N)$, and hence $\injdim_R(N) < \infty$.
	
	(2b) Assume that $r=0$. Then $M=\Gamma_{\fm}(M)$ has finite length, and the assumption $\Hom_R(M,N)\neq 0$ ensures that $d=0$ as $\depth(N) = d$. Let $N''$ be any indecomposable direct summand of $N$. In particular, $N''\neq 0$. If $\Hom_R(M,N'')=0$, then $\Supp(M)\cap \Ass(N'') = \emptyset$, and hence either $M$ or $N''$ is zero as $\Spec(R)=\{\fm\}$, which is a contradiction. Thus $\Hom_R(M,N'')\neq 0$. Applying (1) again, we have $N''\cong\omega$. It follows that $N \cong \omega^u$ for some $u \ge 1$, and $\injdim_R(N) < \infty$.
\end{proof}

\begin{remark}
	In Theorem~\ref{thm:Hom-injdim-finite}, the assumptions $\Hom_R(M,N) \ne 0$	and $\Ext^{i}_R(M,N)=0$ for all $1 \le i \le d-1$
	cannot be omitted, see Examples~\ref{exam:Hom-zero} and \ref{exam:Hom-m-omega} respectively.
\end{remark}

Now we are in a position to prove our first main result.

\begin{theorem}\label{thm:Hom-injdim-finite-consequence}
	Let $M$ and $N$ be nonzero $R$-modules such that $\depth(N)=d$ and $\Ext_R^i(M,N)=0$ for all $1 \le i \le d$. Then $\Hom_R(M,N)$ has finite injective dimension \iff $M$ is free and $N$ has finite injective dimension.
\end{theorem}

\begin{proof}
	The `if' part is trivial. We show the `only if' part. Suppose that $\Hom_R(M,N)$ has finite injective dimension. In view of Remark~\ref{rmk:Hom-nonzero}, $\Hom_R(M,N) \neq 0$. Hence Theorem~\ref{thm:Hom-injdim-finite} implies that $R$ is CM, $N$ is MCM and of finite injective dimension, and $M\cong\Gamma_{\fm}(M)\oplus R^r$ for some $r\ge 0$. It remains to prove that $M$ is free. We consider two possible cases.
	
	(1) Assume that $d\ge 1$. Since $M\cong\Gamma_{\fm}(M)\oplus R^r$, the vanishing $\Ext^d_R(M,N)=0$ yields that $\Ext^d_R(\Gamma_{\fm}(M),N)=0$. Since $N$ is MCM and of finite injective dimension, and $\Gamma_{\fm}(M)$ has finite length, it follows that $\Ext^i_R(\Gamma_{\fm}(M),N)=0$ for all $i \ne d$. Thus $\Ext^i_R(\Gamma_{\fm}(M),N)=0$ for all integers $i$, which implies that $\Gamma_{\fm}(M)=0$ since $N \ne 0$ (see Remark~\ref{rmk:Hom-nonzero}). It follows that $M\cong R^r$.
	
	(2) Suppose that $d=0$. Set $E:=E_R(k)$, the injective hull of $k$. Since both $N$ and $\Hom_R(M,N)$ are injective, we have that $N\cong E^m$ and $\Hom_R(M,N) \cong E^n$ for some $m,n\ge 1$. Therefore, by Matlis duality,
	\begin{center}
		$M^m \cong (M^{\vee\vee})^m \cong ((M^\vee)^m)^\vee \cong \Hom_R(M,N)^{\vee} \cong R^n$.
	\end{center}
	This implies that $M$ is free.
\end{proof}

The examples below show that the assumptions $\Ext^{1\le i\le d-1}_R(M,N)=0$ in Theorem~\ref{thm:Hom-injdim-finite} and $\Ext^{1\le i\le d}_R(M,N)=0$ in Theorem~\ref{thm:Hom-injdim-finite-consequence} cannot be removed.

\begin{example}\label{exam:Hom-m-omega}
	Let $(R,\fm,k)$ be a CM local ring of dimension $d \ge 2$ with a canonical module $\omega$. Then $\Hom_R(\fm,\omega) \cong \omega$ (cf.~\cite[1.2.24]{BH93}) has finite injective dimension. Note that $\omega$ is MCM, but $\fm$ and $\omega$ do not satisfy the Ext vanishing condition as
	\[
	\Ext_R^i(\fm,\omega) \cong \Ext_R^{i+1}(k,\omega) = \left\{\begin{array}{ll}
	0 & \mbox{if } i \ge 1 \mbox{ and } i \neq d-1,\\
	k \neq 0 & \mbox{if } i = d-1.
	\end{array}\right.
	\]
	
	(1) We have $\fm \ncong \Gamma_{\fm}(\fm) \oplus R^r$ for any $r\ge 0$. If possible, suppose $\fm \cong \Gamma_{\fm}(\fm) \oplus R^r$ for some $r\ge 0$. By \cite[Cor.~1.2]{Dut89}, only $\Omega^d_R(k)$ may have a nonzero free direct summand. It follows that $r=0$, otherwise $R$ is a direct summand of $\fm = \Omega_R^1(k)$, which is a contradiction. Therefore $\fm \cong \Gamma_{\fm}(\fm)$ has finite length, which implies that $R$ has finite length, that is again a contradiction as $d \ge 2$.
	
	(2) The $R$-module $\fm$ is not free. Indeed, if $\fm$ is free, then $\projdim_R(k) \le 1$, which is a contradiction as $\dim(R)\ge 2$. If $R$ is non-regular, $\fm$ does not even have finite projective dimension.
\end{example}

\begin{example}\label{exam:M-is-omega*}
	Let $R$ be a $d$-dimensional non-Gorenstein CM normal local ring with a canonical module $\omega$. Set $M=\omega^\ast$ and $N=R$. Then $\Hom_R(M,N) = \omega^{\ast\ast} \cong \omega$ has finite injective dimension. We also have $\depth(N)=d$, but $\injdim_R(R) = \infty$.
\end{example}

The number of vanishing of Ext in Theorem~\ref{thm:Hom-injdim-finite-consequence} cannot be further improved.

\begin{example}\label{exam:Hom-zero}
	Let $(R,\fm,k)$ be a non-Gorenstein CM local ring of dimension $d \ge 1$. Then $\Hom_R(k,R) = 0$, so it has finite injective dimension, $\depth(R) = d$ and $\Ext^i_R(k,R) = 0$ for all $1 \le i \le d-1$, but $\projdim_R(k) = \injdim_R(R) = \infty$.
\end{example}

One may ask the following natural question.

\begin{question}
	Does Theorem~\ref{thm:Hom-injdim-finite-consequence} hold true when $\depth(N)<d$?
\end{question}

Now we discuss the consequences of Theorems~\ref{thm:Hom-injdim-finite} and \ref{thm:Hom-injdim-finite-consequence}.

\begin{corollary}\label{cor:injdim-M*-Gor-criteria}
	Let $M$ be a nonzero $R$-module such that $\Ext_R^i(M,R)=0$ for all $ 1 \le i \le d-1 $ and $ M^* $ has finite injective dimension.
	\begin{enumerate}[\rm (1)]
		\item If $M^* \neq 0$, then $R$ is Gorenstein and $M\cong\Gamma_{\fm}(M) \oplus R^r$ for some $r\ge0$.
		\item If $\Ext_R^d(M,R)=0$, then $R$ is Gorenstein and $M$ is free.
	\end{enumerate}
\end{corollary}

\begin{proof}
	(1) The statement follows from Theorems~\ref{Bass} and \ref{thm:Hom-injdim-finite}.
	
	(2) Note that $M^* \neq 0$ (cf.~Remark~\ref{rmk:Hom-nonzero}). So, by Theorem~\ref{Bass}, $R$ is CM. Hence the assertion follows from Theorem~\ref{thm:Hom-injdim-finite-consequence}.
\end{proof}

\begin{remark}
	The existence of an $R$-module $M$ such that $ M^* $ is nonzero and of finite injective dimension does not necessarily imply that $R$ is Gorenstein, see Example~\ref{exam:M-is-omega*}.
\end{remark}

\begin{corollary}\label{cor:M-star-ARC}
	Let $R$ be a commutative Noetherian ring. Let $M$ be a $($finitely generated$)$ $R$-module such that $ \Hom_R(M,R) $ has finite injective dimension. Then the Auslander-Reiten conjecture holds true for $M$.
\end{corollary}

\begin{proof}
	Suppose that $\Ext_R^{>0}(M,M\oplus R)=0$. What we want to show is that $M$ is projective. It is equivalent to saying that the $R_\fp$-module $M_\fp$ is free for each prime ideal $\fp$ of $R$. Replacing $R$ and $M$ with $R_\fp$ and $M_\fp$ respectively, we may assume that $R$ is local. Hence the assertion follows from Corollary~\ref{cor:injdim-M*-Gor-criteria}.(2).
\end{proof}

\begin{corollary}\label{cor:freeness-self-dual-injdim}
	Let $M$ be an $R$-module such that $\depth(M)=d$, $ \Ext_R^j(M,M) = 0 $ for all $1 \le j \le d$ and $\injdim_R(\Hom_R(M,M)) < \infty$. Then, $M$ is free, and $R$ is Gorenstein.
\end{corollary}

\begin{proof}
	The condition $\depth(M)=d$ particularly ensures that $M$ is nonzero. By Theorem~\ref{thm:Hom-injdim-finite-consequence}, $M$ is free and $\injdim_R(M) < \infty$. So $R$ is Gorenstein.
\end{proof}

We obtain the following criteria for a module to be free over a Gorenstein local ring in terms of vanishing of Ext and projective dimension of Hom.

\begin{corollary}\label{cor:Gor-ARC}
	Suppose that $R$ is Gorenstein. For an $R$-module $M$, the following are equivalent:
	\begin{enumerate}[\rm (1)]
		\item $M$ is free.
		\item $\Hom_R(M,M)$ is free, and $\Ext_R^j(M,M)=0$ for all $ 1 \le j \le d$.
		\item $\Hom_R(M,M)$ has finite projective dimension, and
		\begin{center}
			$ \Ext_R^i(M,R) = \Ext_R^j(M,M) = 0 $ for all $ 1 \le i \le d $ and $ 1 \le j \le d $.
		\end{center}
	\end{enumerate}
\end{corollary}

\begin{proof}
	The implication (1) $\Rightarrow$ (2) is trivial.
	
	(2) $\Rightarrow$ (3): In view of the proof of \cite[Lemma~4.1]{AY98}, we have $\depth(M) = \depth(\Hom_R(M,M))$. Since $R$ is Gorenstein and $\Hom_R(M,M)$ is free, it follows that $M$ is MCM, and hence $ \Ext_R^i(M,R) = 0 $ for all $ 1 \le i \le d $.
	
	(3) $\Rightarrow$ (1): Since $R$ is Gorenstein, $M$ is MCM by \cite[3.5.11]{BH93}. Moreover, for an $R$-module $L$, $ \projdim_R(L) $ is finite if and only if $ \injdim_R(L) $ is finite, cf.~\cite[3.1.25]{BH93}. Therefore $\injdim_R(\Hom_R(M,M))$ is finite. Hence, by Corollary~\ref{cor:freeness-self-dual-injdim}, $M$ is free.
\end{proof}

Next we provide an affirmative answer to the question whether the Auslander-Reiten conjecture holds true if $\Hom_R(M,M)$ has finite injective dimension.

\begin{theorem}\label{thm:main-2}
	Let $M$ be a nonzero $R$-module such that $ \Ext_R^i(M,R) = 0$ and $ \Ext_R^j(M,M) = 0 $		for all $ 1 \le i \le 2d+1 $ and $ 1 \le j \le \max\{1, d-1\} $. Suppose that $\Hom_R(M,M)$ has finite injective dimension. Then, $M$ is free, and $R$ is Gorenstein.
\end{theorem}

\begin{proof}
	Let $\mathbb{F}_M : \; \cdots \stackrel{\partial_3}{\longrightarrow} R^{n_2} \stackrel{\partial_2}{\longrightarrow} R^{n_1} \stackrel{\partial_1}{\longrightarrow} R^{n_0} \stackrel{\partial_0}{\longrightarrow} M \to 0$
	be a minimal free resolution of $M$, and set $\Omega^iM := \Image\partial_i$ for each $i\ge0$.
	As $\Ext_R^1(M,M)=0$, an exact sequence
	\begin{equation*}
		0 \longrightarrow \Hom_R(M,M) \stackrel{f}{\longrightarrow} M^{n_0} \stackrel{g}{\longrightarrow} M^{n_1} \stackrel{h}{\longrightarrow} M^{n_2}
	\end{equation*}
	is induced.	Putting $ N := \coker f = \Image g $ and $ L := \coker g = \Image h $, we have the following exact sequences:
	\begin{align}
	  	& 0 \longrightarrow \Hom_R(M,M) \stackrel{f}{\longrightarrow} M^{n_0} \stackrel{e}{\longrightarrow} N \longrightarrow 0, \label{s.e.s-alpha}\\
	  	& 0 \longrightarrow N \longrightarrow M^{n_1} \longrightarrow L \longrightarrow 0 \label{s.e.s-beta}
	\end{align}
	Note that $N$ is isomorphic to $ \Hom_R(\Omega(M), M) $. As $\Hom_R(M,M)$ is nonzero and of finite injective dimension, by Theorem~\ref{Bass}, $R$ is CM.
	
	We prove the theorem by induction on $ d $.
	
	(1) First, we deal with the case $ d = 0 $. We may assume that $M$ is indecomposable. The module $ \Hom_R(M, M) $ is nonzero and injective. Hence $ \Hom_R(M, M) $ is isomorphic to a finite direct sum of copies of $ E := E_R(k) $. The map $f$ is a split monomorphism. Since $R$ is henselian and $M$ is indecomposable, the Krull--Schmidt theorem yields that $M\cong E$. Hence $\Hom_R(M,M)\cong\Hom_R(E,E)\cong R$. As $\Hom_R(M,M)$ is injective, the Artinian ring $R$ is Gorenstein and thus $ M \cong E \cong R $.
	
	(2) Second, we handle the case $d=1$. We may assume that $R$ is complete, and $M$ is indecomposable. As $R$ is complete, it admits a canonical module $ \omega $.
	
	(2a) We start with the case $ \depth(M) > 0 $. In this case, $M$ is MCM, and so are $ \Hom_R(M,M) $ and $ N \cong \Hom_R(\Omega(M),M) $ as
	\begin{center}
		$ \depth( \Hom_R(X,M) ) \ge \inf \{ 2, \depth M \} > 0 $
	\end{center}
	for any finitely generated $R$-module $X$. It follows that $ \Hom_R(M,M) $ is isomorphic to a finite direct sum of copies of $\omega$, cf.~\cite[3.3.28]{BH93}. The exact sequence \eqref{s.e.s-alpha} splits since it is identified with an element of $\Ext_R^1(N,\Hom_R(M,M))=0$. As $R$ is henselian and $M$ is indecomposable, the Krull--Schmidt theorem implies that $ M \cong \omega $.	We get $ \Hom_R(M, M) \cong R $, and see that $R$ is Gorenstein, and $ M \cong R $.
	
	(2b) Next we consider the case $ \depth M = 0 $. From \eqref{s.e.s-alpha},	an exact sequence
	\begin{equation}\label{s.e.s-Ext}
		\Ext_R^1(L, M^{n_0}) \stackrel{\phi}{\longrightarrow} \Ext_R^1(L, N) \longrightarrow \Ext_R^2(L, \Hom_R(M,M))
	\end{equation}
	is induced, where $ \phi = \Ext_R^1(L,e) $.	The short exact sequence \eqref{s.e.s-beta} can be identified with an element $\beta$ of $\Ext_R^1(L,N)$. As $\injdim_R(\Hom_R(M,M))=d=1<2$, the module $\Ext_R^2(L,\Hom_R(M,M))$ vanishes.	Hence the map $\phi$ in \eqref{s.e.s-Ext} is surjective.	So there exists an element $\gamma\in\Ext_R^1(L,M^{n_0})$ such that $ \phi(\gamma) = \beta $. We obtain a commutative diagram
	$$
	\xymatrixrowsep{8mm} \xymatrixcolsep{8mm}
	\xymatrix{
		& 0 \ar[d] & 0 \ar[d] \\
		& \Hom_R(M,M) \ar@{=}[r]\ar[d]_f & \Hom_R(M,M)\ar[d] \\
		\gamma:0\ar[r] & M^{n_0}\ar[r]\ar[d]_e & Z\ar[r]\ar[d] & L\ar[r]\ar@{=}[d] & 0\\
		\beta:0\ar[r] & N\ar[r]\ar[d] & M^{n_1}\ar[r]\ar[d] & L\ar[r] & 0\\
		& 0 & 0 }
	$$
	with exact rows and columns. Taking the mapping cone of the above chain map $ \gamma \to \beta $, we get an exact sequence
	\begin{equation}\label{s.e.s-from-cone}
		0 \longrightarrow M^{n_0} \longrightarrow N \oplus Z \longrightarrow M^{n_1} \longrightarrow 0,
	\end{equation}
	which is identified with an element of $\Ext_R^1(M^{n_1},M^{n_0})\cong\Ext_R^1(M,M)^{n_1n_0}=0$.	Thus the exact sequence \eqref{s.e.s-from-cone} splits, and an isomorphism $ N \oplus Z \cong M^{n_0+n_1} $ follows. As $R$ is henselian and $M$ is indecomposable, we have $N\cong M^m$ for some $m\ge0$. Set $ r := \type(M) = \dim_k\Hom_R(k,M) > 0 $; recall that $ \depth M = 0 $. There are isomorphisms
	\begin{align*}
		k^{rm} \cong \Hom_R(k,M^m) \cong \Hom_R(k,N) & \cong \Hom_R (k, \Hom_R(\Omega(M), M)) \\
		& \cong \Hom_R (k \otimes_R \Omega(M), M) \cong k^{n_1 r},
	\end{align*}
	whence $ m = n_1 $.	Applying $ k \otimes_R (-) $ to the exact sequence \eqref{s.e.s-alpha}:
	\begin{equation}\label{s.e.s-alpha-2nd}
		0 \longrightarrow \Hom_R(M,M) \stackrel{f}{\longrightarrow} M^{n_0} \longrightarrow M^{n_1} \longrightarrow 0,
	\end{equation}
	we observe that $n_0\ge n_1$.
	
	Fix a minimal prime ideal ${\fp}$ of $R$. If $M_\fp=0$, then of course $M_\fp$ is $R_\fp$-free.
	If	$M_\fp\neq 0$, then applying the induction hypothesis to the Artinian local	ring $R_\fp$ shows that $M_\fp$ is $R_\fp$-free. Thus, in any case, we have	$M_\fp \cong R_\fp^t$ for some $t\ge0$.
	Localization of \eqref{s.e.s-alpha-2nd} at ${\fp}$ gives rise to an exact sequence
	$$
		0 \longrightarrow \Hom_{R_{\fp}}(M_{\fp},M_{\fp}) \stackrel{f_{\fp}}{\longrightarrow} M_{\fp}^{n_0} \longrightarrow M_{\fp}^{n_1} \longrightarrow 0,
	$$
	and we get $tn_0=\rank_{R_{\fp}}M_{\fp}^{n_0}=\rank_{R_{\fp}}\Hom_{R_{\fp}}(M_{\fp},M_{\fp})+\rank_{R_{\fp}}M_{\fp}^{n_1}=t^2+tn_1$.
	Hence $t\in\{0,n_0-n_1\}$. Localizing the resolution $\mathbb{F}_M$, there is an exact sequence
	$$
		0 \longrightarrow (\Omega^2M)_{\fp} \longrightarrow R_{\fp}^{n_1} \longrightarrow R_{\fp}^{n_0} \longrightarrow M_{\fp} \longrightarrow 0,
	$$
	while $M_{\fp}\cong R_{\fp}^t$.
	An isomorphism $(\Omega^2M)_{\fp}\cong R_{\fp}^{n_1-n_0+t}$ follows.
	We have $n_1-n_0+t=0$ when $ t = n_0 - n_1 $, while $ n_1 - n_0 + t = n_1-n_0 \le 0 $ when $t=0$.
	Therefore, $ (\Omega^2M)_{\fp} = 0 $ for every minimal prime ideal ${\fp}$ of $R$.
	Since $\Omega^2(M)$ is a torsion submodule of the torsion-free module $R^{n_1}$, we have $\Omega^2(M)=0$.
	It follows that $ \pd_R M \le 1 < \infty $.	By assumption, $\Ext_R^i(M,R)=0$ for all $1\le i\le 2d+1$.
	In general, when $\pd_RM$ is finite, it is equal to the supremum of integers $i$ such that $\Ext_R^i(M,R)\ne0$.
	Consequently, the module $M$ is free, so is $\Hom_R(M,M)$, and therefore $R$ is Gorenstein.
	
	(3) Finally, we consider the case $d\ge2$. Fix a nonmaximal prime ideal ${\fp}$ of $R$.	Applying the induction hypothesis to $R_\fp$, we see that $M_\fp$ is $R_{\fp}$-free.	Note that $\max\{1,d-1\}=d-1$. We can use \cite[Theorem~1.2.(2)]{KOT21} to observe that $M$ is $R$-free, and so is $\Hom_R(M,M)$, whence $R$ is Gorenstein.
\end{proof}

\begin{corollary}\label{cor:ARC}
	Let $R$ be a commutative Noetherian ring. Let $M$ be a $($finitely generated$)$ $R$-module such that $ \Hom_R(M,M) $ has finite injective dimension. Then the Auslander-Reiten conjecture holds true for $M$.
\end{corollary}

\begin{proof}
	Let $ \Ext_R^{>0}(M, M \oplus R) = 0 $. We show that $M$ is projective. It is equivalent to showing that the $R_\fp$-module $M_\fp$ is free for each prime ideal $\fp$ of $R$. Replacing $R$ and $M$ with $R_\fp$ and $M_\fp$ respectively, we may assume that $R$ is local. Thus the desired statement follows from Theorem~\ref{thm:main-2}.
\end{proof}

In view of Theorem~\ref{thm:main-2}, we may wonder if $\Hom_R(M,M)$ having finite injective
dimension implies that $M$ has finite projective dimension. This is not true even when $R$ is a Gorenstein normal local ring. Particularly, it shows that the hypothesis on the vanishing of Ext in Theorem~\ref{thm:main-2} cannot be omitted.

\begin{example}
	Let $R$ be a Gorenstein normal local ring, and $I$ be a nonzero ideal of $R$. Then $\Hom_R(I,I) \cong R$. Therefore $\Hom_R(I,I)$ has finite injective dimension, but $I$ does not necessarily have finite projective dimension. For example, one may consider $R=k[[x,y,z]]/(x^2-yz)$ with $k$ a field, and $I=\fm=(x,y,z)$.
\end{example}

We close this section with the following natural question.

\begin{question}\label{ques:Hom-inj-dim-finite-R-Gor}
	If there exists a nonzero $R$-module $M$ such that $\Hom_R(M,M)$ has finite injective dimension, then is $R$ Gorenstein?
\end{question}


\section{Characterizations of Gorenstein local rings}

In this section, we provide a number of characterizations of Gorenstein local rings in terms of finite injective dimension of certain Hom. We start with answering Question~\ref{ques:Hom-inj-dim-finite-R-Gor} affirmatively in some special cases.

\begin{proposition}
	Let $M$ be an $R$-module. Suppose that {\rm (i)} $M$ is torsion-free, {\rm (ii)} $M$ is locally free in codimension $1$, {\rm (iii)} $M$ has a rank, and {\rm (iv)} $\rank(M)$ is invertible in $R$. If $\Hom_R(M,M)$ has finite injective dimension, then $R$ is Gorenstein.
\end{proposition}

\begin{proof}
	Since $M$ is nonzero, so is $\Hom_R(M,M)$. Then Theorem~\ref{Bass} yields that $R$
	is CM. It follows from \cite[A.2 and A.5]{HL04} that $R$ is a direct summand of $\Hom_R(M,M)$. Hence $R$ has finite injective dimension, i.e., $R$ is Gorenstein.
\end{proof}

\begin{proposition}\label{prop:injdim-Hom-Gor}
	Suppose that $R$ admits a module $M$ such that
	\begin{center}
		$\depth(\Hom_R(M,M))=\depth(R)$ and $\injdim_R(\Hom_R(M,M)) < \infty$.
	\end{center}
	Then $R$ is Gorenstein.
\end{proposition}

\begin{proof}
	Replacing $R$ with its completion, we may assume that $R$ is complete. By Theorem~\ref{Bass}, $R$ is CM, and hence $\Hom_R(M,M)$ is MCM.	Since $R$ is complete CM, it admits a canonical module $\omega$, and we have that $\Hom_R(M,M) \cong \omega^n$	for some $n \ge 1$, cf.~\cite[3.3.28]{BH93}. It follows that
	\begin{align*}\label{Hom-free}
	R^{n^2} \cong \Hom_R(\omega^n,\omega^n)
	&\cong \Hom_R(\Hom_R(M,M),\Hom_R(M,M))\\
	&\cong \Hom_R(M \otimes_R \Hom_R(M,M), M).
	\end{align*}
	Consider the $R$-module homomorphisms
	\begin{center}
		$f: M\to M\otimes_R \Hom_R(M,M)$ \; and \; $g:M\otimes_R \Hom_R(M,M) \to M$
	\end{center}
	defined by $f(x)=x \otimes \id_M$ and $g(x \otimes h) = h(x)$ respectively. Clearly, the composition $g \circ f = \id_M$, and hence $f$ is a split monomorphism. Therefore $M$ is a direct summand of $M\otimes_R \Hom_R(M,M)$. This implies that $\omega^n \cong \Hom_R(M,M)$ is a direct summand of $\Hom_R(M\otimes_R \Hom_R(M,M), M) \cong R^{n^2}$. So, by the Krull--Schmidt theorem, $\omega \cong R$, and hence $R$ is Gorenstein.
\end{proof}

%
%

\begin{para}
	A partial positive answer to Question~\ref{ques:Hom-inj-dim-finite-R-Gor} is provided in \cite[Cor.~4.5]{CGP21}, where it is shown that if there exists a nonzero $R$-module $M$ of depth $\ge d - 1$ such that the injective dimensions of $M$, $\Hom_R(M,M)$ and $\Ext_R^1(M,M)$ are finite, then the projective dimension of $M$ is finite, and $R$ is Gorenstein.
\end{para}

\begin{para}\label{para:char-Gor}
	Let $R$ be CM. In \cite[Thm.~3.1]{Ulr84}, Ulrich gave a criterion for $R$ to be Gorenstein: If there is an $ R $-module $ L $ of positive rank such that $ 2 \nu(L) > e(R) \rank(L) $ and $ \Ext_R^{1 \le i \le d}(L,R) = 0 $, then $ R $ is Gorenstein. Hanes-Huneke and Jorgensen-Leuschke gave some analogous criteria in \cite[Thms.~2.5 and 3.4]{HH05} and
	\cite[Thms.~2.2 and 2.4]{JL07} respectively. Recently, Lyle and Monta\~{n}o \cite[Thm.~D]{LM20} showed that a generically Gorenstein CM local ring $R$ that has a canonical module is Gorenstein if there is an MCM $R$-module $L$ such that $ \Ext_R^{1 \le i \le d+1}(L,R) = 0 $ and $e_R(L) \le 2 \nu(L)$.
\end{para}

\begin{para}\label{para:char-reg}
	Like Theorems~\ref{PS} and \ref{Foxby}, having finite injective dimension of certain modules also ensures that the base ring is regular. For example, $R$ is regular \iff its residue field $k$ has finite injective dimension, see, e.g., \cite[3.1.26]{BH93}. More generally, it is shown in \cite[Thm.~3.7]{GGP} that $R$ is regular if and only if some syzygy $\Omega_R^n(k)$ $(n \ge 0)$ has a nonzero direct summand of finite injective dimension.
\end{para}

Inspired by the results mentioned in \ref{para:char-Gor}, \ref{para:char-reg} and Theorems~\ref{PS} and \ref{Foxby}, we obtain the following new characterizations of Gorenstein local rings in terms of vanishing of certain Ext and finite injective dimension of Hom.

\begin{theorem}\label{thm:characterizations-Gor}
	The following statements are equivalent:
	\begin{enumerate}[\rm (1)]
		\item $R$ is Gorenstein.
		\item $R$ admits a module $M$ such that $\Ext_R^i(M,R)=0$ for all $ 1 \le i \le d-1 $ and $ M^* $ is nonzero and of finite injective dimension.
		\item $R$ admits a nonzero module $M$ such that $\Ext_R^i(M,R)=\Ext_R^j(M,M)=0$ for all $1\le i\le 2d+1$ and $1\le j\le\max\{1,d-1\}$, and $\injdim_R(\Hom_R(M,M)) < \infty$.
		\item $R$ admits a module $M$ such that $\depth(M)=d$, $\Ext_R^j(M,M)=0$ for all $1 \le j \le d$ and $\injdim_R(\Hom_R(M,M)) < \infty$.
		\item $R$ admits a module $M$ such that
		\begin{center}
			$ \depth(\Hom_R(M,M)) = \depth(R) $ and $ \injdim_R(\Hom_R(M,M)) < \infty $.
		\end{center}
	\end{enumerate}
\end{theorem}

\begin{proof}
	The implications (1) $\Rightarrow$ (2), (3), (4) and (5) are trivial as $M=R$ satisfies the respective conditions. The reverse implications (2) $\Rightarrow$ (1), (3) $\Rightarrow$ (1), (4) $\Rightarrow$ (1) and (5) $\Rightarrow$ (1) follow from \ref{cor:injdim-M*-Gor-criteria}.(1), \ref{thm:main-2}, \ref{cor:freeness-self-dual-injdim} and \ref{prop:injdim-Hom-Gor} respectively.
\end{proof}

\begin{remark}\label{rmk:Holm}
	This is known due to Holm \cite{Ho04} that if there exists a nonzero $R$-module $M$ of finite Gorenstein dimension and finite injective dimension, then $R$ is Gorenstein. It should be noted that Theorem~\ref{thm:characterizations-Gor}.(1) $\Leftrightarrow$ (2) would not follow from the result of Holm. Moreover, we can recover that if $M$ is a nonzero $R$-module such that $\gdim_R(M)=0$ and $\injdim_R(M) < \infty$, then $M^*$ satisfies the condition (2) in \ref{thm:characterizations-Gor} as $\Ext_R^i(M^*,R)=0$ for all $i\ge 1$ and $(M^*)^* \cong M$ has finite injective dimension, hence $R$ is Gorenstein.
\end{remark}

\section*{Acknowledgments}
Ghosh was supported by Start-up Research Grant (SRG) from SERB, DST, Govt.~of India with the Grant No SRG/2020/000597. Takahashi was partly supported by JSPS Grant-in-Aid for Scientific Research 19K03443.


\begin{thebibliography}{AAAA}
	
	\bibitem{Ar09}
	T.~Araya, {\it The Auslander-Reiten conjecture for Gorenstein rings}, Proc. Amer. Math. Soc. {\bf 137} (2009), 1941--1944.
	
	\bibitem{AY98} T.~Araya and Y.~Yoshino, {\it Remarks on a depth formula, a grade inequality and a conjecture of Auslander}, Comm. Algebra {\bf 26} (1998), 3793--3806.
	
	
	\bibitem{ADS93} M.~Auslander, S.~Ding and O.~Solberg, {\it Liftings and weak liftings of modules}, J. Algebra {\bf 156}, (1993), 273--317.
	
	\bibitem{AR75} M.~Auslander and I.~Reiten, {\it On a generalized version of the Nakayama conjecture}, Proc. Amer. Math. Soc. {\bf 52} (1975), 69--74.
%
%
%
%
	
	\bibitem{BH93} W.~Bruns and J.~Herzog, \emph{Cohen-{M}acaulay rings}, Cambridge Studies in Advanced Mathematics, vol.~39, Cambridge University Press, Cambridge, 1993.
%
%
	
%
	
	\bibitem{DEL21} H.~Dao, M.~Eghbali and J.~Lyle, {\it Hom and Ext, revisited}, J. Algebra {\bf 571} (2021), 75--93.
	
	\bibitem{DG21} S.~Dey and D.~Ghosh, {\it Complexity and rigidity of Ulrich modules, and some applications}, Math. Scand. (to appear), \href{https://arxiv.org/pdf/2201.00984.pdf}{arXiv:2201.00984}.
	
%
	
	\bibitem{Dut89} S. P. Dutta, {\it Syzygies and homological conjectures}, in: Commutative Algebra, Berkeley, CA, 1987, in: Math. Sci. Res. Inst. Publ., vol. 15, Springer, New York, 1989, pp. 139--156.
	
	\bibitem{GGP} D.~Ghosh, A.~Gupta and T.J.~Puthenpurakal, {\it Characterizations of regular local rings via syzygy modules of the residue field}, J. Commut. Algebra {\bf 10} (2018), 327--337.
	
	\bibitem{CGP21} D.~Ghosh and T.J.~Puthenpurakal, {\it Gorenstein rings via homological dimensions, and symmetry in vanishing of Ext and Tate cohomology}, \href{https://arxiv.org/pdf/2302.06267.pdf}{arXiv:2302.06267}
	
	
	\bibitem{GT17} S.~Goto and R.~Takahashi, {\it On the Auslander-Reiten conjecture for Cohen-Macaulay local rings}, Proc. Amer. Math. Soc. 145 (2017), no. 8, 3289–3296. 
	
	\bibitem{HH05} D. Hanes and C. Huneke, {\it Some criteria for the Gorenstein property}, J. Pure Appl. Algebra {\bf 201} (2005), 4--16.

	\bibitem{Ho04} H.~Holm, {\it Rings with finite Gorenstein injective dimension}, Proc. Amer. Math. Soc. {\bf 132} (2004), 1279--1283.
	
	\bibitem{HL04} C.~Huneke and G.J.~Leuschke, {\it On a conjecture of Auslander and Reiten}, J.  Algebra {\bf 275} (2004), 781--790.
		
	\bibitem{JL07} D.A.~Jorgensen and G.J.~Leuschke, {\it On the growth of the Betti sequence of the canonical module}, Math. Z. {\bf 256} (2007), 647--659.
%
%
%
	
	\bibitem{Fo77} H.-B.~Foxby, {\it Isomorphisms between complexes with applications to the homological theory of modules}, Math. Scand. {\bf 40} (1977), 5--19.
	
%
%
%
	
	\bibitem{KOT21} K.~Kimura, Y.~Otake and R.~Takahashi, {\it Maximal Cohen-Macaulay tensor products and vanishing of Ext modules}, \href{https://arxiv.org/pdf/2106.08583v1.pdf}{arXiv:2106.08583}.
	
	\bibitem{LM20} J.~Lyle and J.~Monta\~{n}o, {\it Extremal growth of Betti numbers and trivial vanishing of (co)homology}, Trans. Amer. Math. Soc. {\bf 373} (2020), 7937--7958.

%
	
	\bibitem{NS17} S.~Nasseh and S.~Sather-Wagstaff, {\it Vanishing of Ext and Tor over fiber products}, Proc. Amer. Math. Soc. {\bf 145} (2017), 4661--4674.
	
	\bibitem{PS73} C.~Peskine, L.~Szpiro, {\it Dimension projective finie et cohomologie locale. Applications \`{a} la d\'{e}monstration de conjectures de M.~Auslander, H.~Bass et A.~Grothendieck}, Inst. Hautes \'{E}tudes Sci. Publ. Math., No. 42, (1973), 47--119.
	
	\bibitem{Rob87} P. Roberts, {\it Le th\'{e}or\`{e}me d'intersection}, C.R. Acad. Sci. Paris 304 (1987), 177--180.
	
	\bibitem{ST19} A.~Sadeghi and R.~Takahashi, {\it Two generalizations of Auslander-Reiten duality and applications}, Illinois J. Math. {\bf 63} (2019), 335--351.
	
%
%

	\bibitem{Ulr84} B. Ulrich, {\it Gorenstein rings and modules with high numbers of generators}, Math. Z. {\bf 188} (1984), 23--32.
	
%
	
\end{thebibliography}
\end{document}